\documentclass{amsart}
\usepackage{amssymb}
\usepackage{graphicx}
\usepackage{amsfonts}
\usepackage{amssymb}
\usepackage{amsmath}
\usepackage{amsthm}
\usepackage{enumerate}
\usepackage{tabularx}
\usepackage{centernot}
\usepackage{mathtools}
\usepackage{stmaryrd}
\usepackage{amsthm,amssymb}
\usepackage{etoolbox}
\usepackage{tikz}
\usepackage{amssymb}
\usetikzlibrary{matrix}
\usepackage{tikz-cd}
\usepackage{tikz}
\usepackage{marginnote}
\definecolor{mygray}{gray}{0.85}
\usepackage[backgroundcolor=mygray,colorinlistoftodos,prependcaption,textsize=small]{todonotes}
\usepackage{xargs}                      

\renewcommand{\leq}{\leqslant}

\makeatletter
\def\subsection{\@startsection{subsection}{3}%
  \z@{.5\linespacing\@plus.7\linespacing}{.3\linespacing}%
  {\bfseries\centering}}
\makeatother

\makeatletter
\def\subsubsection{\@startsection{subsubsection}{3}%
  \z@{.5\linespacing\@plus.7\linespacing}{.3\linespacing}%
  {\centering}}
\makeatother

\makeatletter
\def\myfnt{\ifx\protect\@typeset@protect\expandafter\footnote\else\expandafter\@gobble\fi}
\makeatother

\newtheorem{theorem}{Theorem}

\newtheorem{corollary}[theorem]{Corollary}
\newtheorem{definition}[theorem]{Definition}
\newtheorem{lemma}[theorem]{Lemma}

\newtheorem{fact}[theorem]{Fact}
\newtheorem{strategy}[theorem]{Strategy}
\newtheorem{remark}[theorem]{Remark}
\newtheorem{notation}[theorem]{Notation}

\newtheorem{oproblem}[theorem]{Open Problem}
\newtheorem{construction}[theorem]{Construction}

\newcounter{claimcounter}
\numberwithin{claimcounter}{theorem}

\begin{document}

\begin{abstract} For every infinite graph $\Gamma$ we construct a non-Desarguesian projective plane $P^*_{\Gamma}$ of the same size as $\Gamma$ such that $Aut(\Gamma) \cong Aut(P^*_{\Gamma})$ and $\Gamma_1 \cong \Gamma_2$ iff $P^*_{\Gamma_1} \cong P^*_{\Gamma_2}$. Furthermore, restricted to structures with domain $\omega$, the map $\Gamma \mapsto P^*_{\Gamma}$ is Borel. On one side, this shows that the class of countable non-Desarguesian projective planes is Borel complete, and thus not admitting a Ulm type system of invariants. On the other side, we rediscover the main result of \cite{projective} on the realizability of every group as the group of collineations of some projective plane. Finally, we use classical results of projective geometry to prove that the class of \mbox{countable Pappian projective planes is Borel complete.}
\end{abstract}

\title{The Class of Non-Desarguesian Projective Planes is Borel Complete}
\thanks{Partially supported by European Research Council grant 338821.}


\author{Gianluca Paolini}
\address{Einstein Institute of Mathematics,  The Hebrew University of Jerusalem, Israel}

\maketitle

\section{Introduction}

	
	\begin{definition}\label{def_plane} A {\em plane} is a system of points and lines satisfying:
	\begin{enumerate}[(A)]
	\item every pair of distinct points determines a unique line;
	\item every pair of distinct lines intersects in at most one point;
	\item every line contains at least two points;
	\item there exist at least three non-collinear points.
\end{enumerate}
A plane is {\em projective} if in addition:
	\begin{enumerate}[(B')]
	\item every pair of lines intersects in exactly one point.
\end{enumerate}
\end{definition}
	
	As well-known (see e.g. \cite{rota} and \cite[pg. 148]{kung}), the class of planes (resp. projective planes) corresponds canonically to the class of simple rank $3$ matroids (resp. simple modular rank $3$ matroids), or, equivalently, to the class of geometric lattices of rank $3$ (resp. modular geometric lattices of rank $3$). We prove:

	\begin{theorem}\label{main_theorem} For every graph $\Gamma = (V, E)$ there exists a plane $P_{\Gamma}$ such that:
	\begin{enumerate}[(1)]
	\item if $\Gamma$ is finite, then $P_{\Gamma}$ has size $3|V| + |E| + 17$;
	\item if $\Gamma$ is infinite, then $P_{\Gamma}$ has the same size of $\Gamma$;
	\item except for $17$ points, every point of  $P_{\Gamma}$ is incident with at most two non-trivial lines;
	\item\label{auto}  $Aut(\Gamma) \cong Aut(P_{\Gamma})$;
	\item\label{iso_invariance} $\Gamma_1  \cong \Gamma_2$ if and only if $P_{\Gamma_1} \cong P_{\Gamma_2}$;
	\item restricted to structures with domain $\omega$, the map $\Gamma \mapsto P_{\Gamma}$ is Borel (with respect to the naturally associated Polish topologies).
\end{enumerate}
\end{theorem}

	We then combine (a modification of) the construction $\Gamma \mapsto P_{\Gamma}$ of Theorem \ref{main_theorem} with the the map $P \mapsto F(P)$ associating to each plane its free projective extension (in the sense of \cite{hall_proj}, cf. also Definition \ref{def_free_ext}), and prove:

\begin{theorem}\label{main_theorem_proj} For every infinite graph $\Gamma$ there exists a projective \mbox{plane $P^*_{\Gamma}$ such that:}
	\begin{enumerate}[(1)]
	\item $P^*_{\Gamma}$ has the same size of $\Gamma$;
	\item $P^*_{\Gamma}$ is non-Desarguesian;
	\item $Aut(\Gamma) \cong Aut(P^*_{\Gamma})$;
	\item $\Gamma_1  \cong \Gamma_2$ if and only if $P^*_{\Gamma_1} \cong P^*_{\Gamma_2}$;
	\item restricted to structures with domain $\omega$, the map $\Gamma \mapsto P^*_{\Gamma}$ is Borel (with respect to the naturally associated Polish topologies).
\end{enumerate}
\end{theorem}	

	As a first consequence we get:
	
		\begin{definition}\label{def_classes}
	\begin{enumerate}[(1)]
	\item We say that a plane is simple (or $17$-simple) if except for $17$ points every point is incident with at most two non-trivial lines.
	\item We denote by $\mathbf{K}_1$ the class of countable simple planes.
	\item We denote by $\mathbf{K}_2$ the class of countable non-Desarguesian projective planes.
\end{enumerate}
\end{definition}

\begin{corollary}\label{main_cor2} Let $\mathbf{K}$ be either $\mathbf{K}_1$ or $\mathbf{K}_2$ (cf. Definition \ref{def_classes}). Then:
\begin{enumerate}[(1)]
	\item $\mathbf{K}$ is Borel complete (i.e. the isomorphism relation on $\mathbf{K}$ is $Sym(\omega)$-complete);
	\item $\mathbf{K}$ does not admit a Ulm type classification (cf. \cite{ulm_inv_paper} for this notion).
\end{enumerate}
\end{corollary}

	In \cite{frucht} and \cite{frucht2} Frucht showed that every finite group is the group of automorphisms of a finite graph. Later, Sabadussi \cite{sabi} and, independently, de Groot \cite{groot} proved that every group is the group of automorphisms of a graph. Using this, Harary, Piff, and Welsh \cite{piff} proved that every group is the group of automorphisms of a graphic matroid, possibly of infinite rank. In \cite{bonin}, Bonin and Kung showed that every infinite group is the group of automorphisms of a Dowling plane of the same cardinality. In \cite{projective}, Mendelsohn proved that every group is the group of collineations of some projective plane. Using Theorems \ref{main_theorem} and \ref{main_theorem_proj} we rediscover and improve these results:

	\begin{corollary}\label{main_cor} 
	\begin{enumerate}[(1)]
	\item For every finite structure $M$ (in the sense of model theory) there exists a simple plane $P_M$ such that $P_M$ is finite and $Aut(P_M) \cong Aut(M)$.
	\item For every infinite structure $M$ (in the sense of model theory) there exists a simple plane $P_M$ such that $|M| = |P_M|$ and $Aut(P_M) \cong Aut(M)$.
	\item For every infinite structure $M$ there exists a non-Desarguesian projective plane $P_M$ such that $|M| = |P_M|$ and $Aut(P_M) \cong Aut(M)$.
\end{enumerate}	
\end{corollary}

	Finally, we use classical results of projective geometry to prove:
	
		\begin{theorem}\label{des_theorem} Let $\mathbf{K}_3$ be the class of countable Pappian\footnote{Notice that Pappian planes are Desarguesian.} projective planes. Then:
\begin{enumerate}[(1)]
	\item $\mathbf{K}_3$ is Borel complete;
	\item $\mathbf{K}_3$ does not admit a Ulm type classification.
\end{enumerate}
\end{theorem}

We leave the following open problem:
	
	\begin{oproblem} Characterize the Lenz-Barlotti classes of countable projective planes which are Borel complete.
\end{oproblem}

	
\section{Preliminaries}\label{preliminaries}
	
	Given a plane $P$ we will freely refer to the canonically associated geometric lattice $G(P)$. On this see e.g. \cite{rota}, or \cite[Section 2]{paolini&hyttinen}, for an introduction directed to logicians. For our purposes the \mbox{lattice-theoretic definitions in Definition \ref{basic_def}(\ref{sup}-\ref{inf}) suffice.}
	
	\begin{definition}\label{basic_def} Let $P$ be a plane.
	\begin{enumerate}[(1)]
	\item\label{sup} Given two distinct points $a_1$ and $a_2$ of $P$ we let $a_1 \vee a_2$ be the unique line that they determine.
	\item\label{inf} Given two distinct lines $\ell_1$ and $\ell_2$ of $P$ we let $\ell_1 \wedge \ell_2$ be the unique point in their intersection, if such a point exists, and $0$ otherwise.
	\item The {\em size} $|P|$ of a plane $P$ is the size of its set of points. 
	\item We say that the point $a$ (resp. the line $\ell$) is {\em incident} with the line $\ell$ (resp. the point $a$) if the point $a$ (resp. the line $\ell$) is contained in the line $\ell$ (resp. contains the point $a$).
	\item\label{trivial} We say that the line $\ell$ from $P$ is {\em trivial} if $\ell$ is incident with exactly two points from $P$.
	\item We say that two lines $\ell_1$ and $\ell_2$ from $P$ are {\em parallel} in $P$ if $\ell_1 \wedge \ell_2 = 0$, i.e. there is no point $p \in P$ incident with both $\ell_1$ and $\ell_2$. 
	\item We say that three distinct points $a_1, a_2, a_3$ of $P$ are {\em collinear} if there is a line $\ell$ in $P$ such that $a_i$ is incident with $\ell$ for every $i = 1, 2, 3$ (in this case we also say that the set $\{a_1, a_2, a_3\}$ is dependent).
\end{enumerate}
\end{definition}
	
	We will use crucially the following fact from the theory of one-point extensions of matroids from \cite{crapo} (see also \cite[Chapter 10]{rota} and \cite[Theorem 2.12]{paolini&hyttinen}).

	\begin{fact}\label{fact} Let $P$ be a plane, $L$ a set of parallel lines of $P$ (in particular $L$ can be empty or a singleton) and $p \not\in P$. Then there exists a plane $P(L)$ (unique modulo isomorphism) such that its set of points is the set of points of $P$ plus the point $p$, and $p, q, r$ are collinear in $P(L)$ if and only if $q \vee r \in L$.
\end{fact}	

	We now introduce Hall's notion of free projective extension from \cite{hall_proj}. In exposition and results we follow \cite[Chapter XI]{piper}.

	\begin{definition}[Cf. {\cite[Theorem 11.4]{piper}}]\label{def_free_ext} Given a plane $P$ we define by induction on $n < \omega$ a chain of planes $(P_n : n < \omega)$ as follows:
\newline $n = 0$. Let $P_n = P$.
\newline $n = m+1$. For every pair of parallel lines $\ell \neq \ell'$ in $P_m$ add a new point $\ell \wedge \ell'$ to $P_m$ incident \mbox{with only $\ell$ and $\ell'$. Let $P_n$ be the resulting plane.}
\newline We define the {\em free projective extension} of $P$ to be $F(P) : = \bigcup_{n < \omega} P_n$.
\end{definition}

	\begin{definition} Given two planes $P_1$ and $P_2$, we say that $P_1$ is a {\em subplane} of $P_2$ if $P_1 \subseteq P_2$, points of $P_1$ are points of $P_2$, lines of $P_1$ are lines of $P_2$, and the point $p$ is on the line $\ell$ in $P_1$ if and only if the point $p$ is on the line $\ell$ in $P_2$.
\end{definition}

	\begin{definition}\label{def_conf} Let $P$ be a plane.
	\begin{enumerate}[(1)]
	\item If $P$ is {\em finite}, then we say that $P$ is {\em confined} if every point of $P$ is incident with at least three lines of $P$, and every line of $P$ is non-trivial (cf. Definition~\ref{basic_def}(\ref{trivial})).
	\item We say that $P$ is confined if every point and every line of $P$ is contained in a finite confined subplane of $P$.
\end{enumerate}	
\end{definition}

	We will make a crucial use of the following facts:
	
	\begin{definition}[{\cite[Definition 5.1.1]{steven}}]\label{desarg_def} Let $P$ be a projective plane. We say that $P$ is {\em Desarguesian} if given two triples of distinct points $p, q, r$ and $p', q', r'$, if the lines $p\vee p'$, $q \vee q'$ and $r \vee r'$ are incident with a common point, then the points $(p \vee q) \wedge (p' \vee q')$, $(p \vee r) \wedge (p' \vee r')$ and $(q \vee r) \wedge (q' \vee r')$ are collinear.
\end{definition}
	
	\begin{fact}[{\cite[Theorem 4.6]{hall_proj}}]\label{desargue_fact} Let $P$ be a plane which is not a projective plane. Then $F(P)$ is non-Desarguesian.
\end{fact}

	\begin{fact}[{\cite[Theorem 11.11]{piper}}]\label{piper_fact1} Le $P_1$ and $P_2$ be confined planes. Then the following are equivalent:
	\begin{enumerate}[(1)]
	\item $F(P_1) \cong F(P_2)$;
	\item $P_1 \cong P_2$.
\end{enumerate}
\end{fact}

	\begin{fact}[{\cite[Theorem 11.18]{piper}}]\label{piper_fact2} Let $P$ be a confined plane. Then:
	$$Aut(P) \cong Aut(F(P)).$$
\end{fact}

	The following facts are classical results of projective geometry.
	
	\begin{definition}[{\cite[Definition 6.1.1]{steven}}]\label{pappian_def} Let $P$ be a projective plane. We say that $P$ is {\em Pappian} if given two triples of distinct collinear points $p, q, r$ and $p', q', r'$ on distinct lines $\ell$ and $\ell'$, respectively, if $\ell \wedge \ell'$ is different from all six points, then the points $(p \vee q') \wedge (p' \vee q)$, $(p \vee r') \wedge (p' \vee r)$ and $(q \vee r') \wedge (q' \vee r)$ are collinear.
\end{definition}

	\begin{definition} Given a field $K$ we denote by $\mathfrak{P}(K)$ the corresponding projective plane (cf. e.g. \cite[Section 2]{piper}).
\end{definition}

	\begin{fact}[{\cite[Theorem 2.6]{piper}}]\label{pappian_field_fact} Let $K$ be a field. Then $\mathfrak{P}(K)$ is Pappian.
\end{fact}

	\begin{fact}[{\cite[Theorem 2.8]{piper}}]\label{pappian_fact} Let $K$ and $K'$ be fields. Then $\mathfrak{P}(K) \cong \mathfrak{P}(K')$ if and only if $K \cong K'$.
\end{fact}

	Concerning the topological notions occurring in Theorem \ref{main_theorem}, they are in the sense of invariant descriptive set theory of $\mathfrak{L}_{\omega_1, \omega}$-classes, see e.g. \cite[Chapter 11]{gao_invariant} for a thorough introduction.  Notice that the classes of planes, simple planes, projective planes, (non-)Desarguesian projective planes (cf. Definition \ref{desarg_def}), and Pappian planes (cf. Definition \ref{pappian_def}) are first-order classes, considered e.g. in a language specifying points, lines and the point-line incidence relation. 

\section{Proof of Theorem \ref{main_theorem}}

	In this section we prove Theorem \ref{main_theorem}.

\begin{notation}\label{notation_Bonin_paper} We denote by $P_*$ the plane represented in Figure \ref{figure1}. The plane $P_*$ is taken from \cite{bonin}, where it is denoted as $T_S$ for $S = \{ 0, 1, 2, 3 \}$.
\begin{figure}[htb]
\begin{center}
\includegraphics[clip, trim=0.5cm 10.5cm 0.5cm 10.5cm, width=1.20\textwidth]{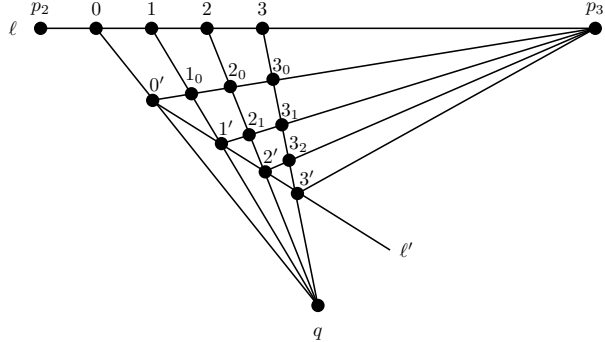}\caption{The plane $P_*$.}\label{figure1}
\end{center}
\end{figure}
\end{notation}
		
	\begin{proof}[Proof of Theorem \ref{main_theorem}] Let $\Gamma = (V, E)$ be given and let $\{ v_{\alpha} : \alpha < \lambda \}$ list $V$ without repetitions. For $\gamma \leq \lambda$, let $\Gamma_{\gamma} = (V_{\gamma}, E_{\gamma})$ be such that $V_{\gamma} = \{ v_{\beta} : \beta < \gamma \}$ and for ${\alpha} < \beta < \gamma$ we have $v_{\alpha} E_{\gamma} v_{\beta}$ if and only if $v_{\alpha} E v_{\beta}$. Let $P_*$ be the plane from Notation \ref{notation_Bonin_paper}. Notice that $|P_*| = 17$ and, as proved in \cite[Lemma 2]{bonin}, $P_*$ is rigid, i.e. $Aut(P_*) = \{ e \}$.
	
\smallskip
\noindent By induction on $\beta \leq \lambda$, we construct a plane $P_{\Gamma}(\beta)$ such that its set of points is:
	\begin{equation}\label{equation_points} \tag{$*$}
	 P_* \cup \{ p_{(\alpha, 0)} : \alpha < \beta \} \cup \{ p_{(\alpha, 1)} : \alpha < \beta \} \cup \{ p_{(\alpha, 2)} : \alpha < \beta \} \cup \{ p_e : e \in E_{\beta} \}.
\end{equation}
For $\beta = 0$, let $P_{\Gamma}(\beta) = P_*$. For $\beta$ limit ordinal, let $P_{\Gamma}(\beta) = \bigcup_{\alpha < \beta} P_{\Gamma}(\alpha)$. For $\beta = \alpha + 1$, we construct $P_{\Gamma}(\beta)$ from $P_{\Gamma}(\alpha)$ via a sequence of one-point extensions as follows. Firstly, add a new point $p_{(\alpha, 0)}$ under the line $p_2 \vee 1'$ (using Fact \ref{fact} with $L = \{p_2 \vee 1'\}$). Secondly, add a new point $p_{(\alpha, 1)}$ under the line $0 \vee 1'$ (using Fact \ref{fact} with $L = \{0 \vee 1'\}$). Thirdly, add a new point $p_{(\alpha, 2)}$ under the line $p_{(\alpha, 0)} \vee p_{(\alpha, 1)}$ (using Fact \ref{fact} with $L = \{ p_{(\alpha, 0)} \vee p_{(\alpha, 1)}\}$). Fourthly, for every $e = \{ v_{\delta}, v_{\alpha} \} \in E_{\beta}$ add a point $p_e$ under the parallel lines $p_{(\delta, 0)} \vee p_{(\delta, 1)}$ and $p_{(\alpha, 0)} \vee p_{(\alpha, 1)}$ (using Fact \ref{fact} with $L = \{ p_{(\delta, 0)} \vee p_{(\delta, 1)},  p_{(\alpha, 0)} \vee p_{(\alpha, 1)} \}$). \mbox{Let $P_{\Gamma}(\beta)$ be the resulting plane.}

\smallskip
\noindent Let $P_{\Gamma}(\lambda) = P_{\Gamma}$. First of all, by (\ref{equation_points}), the size of $P_{\Gamma}$ is clearly as wanted. Also, if $p \notin P_*$, then, by construction, $p$ is incident with at most two non-trivial lines. Furthermore, the construction of $P_{\Gamma}$ from $\Gamma$ is explicit, and so, restricted to structures with domain $\omega$, the map $\Gamma \mapsto P_{\Gamma}$ is easily seen to be Borel, since to know a finite substructure of $P_\Gamma$ it is enough to know a finite part of $\Gamma$.  Thus, we are only left to show items (\ref{auto}) and (\ref{iso_invariance}) of the statement of the theorem. To this extent, first of all notice that, letting $p_{(\alpha, 0)} \vee p_{(\alpha, 1)} = \ell_{\alpha}$ (for $\alpha < \lambda$), we have: 
\begin{enumerate}[$(\star_1)$]
\item the set of lines $\{ \ell_{\alpha} : \alpha < \lambda \}$ of $P_{\Gamma}$ with edge relation $\ell_\alpha E \ell_\beta$ if and only if $\ell_\alpha \wedge \ell_\beta \neq 0$ (i.e. the two lines intersect) is isomorphic to $\Gamma$.
\end{enumerate}
Now, for a point $p$ let $\varphi(p)$ be the following statement: 
\begin{enumerate}[$(S)$]
\item $p$ is incident with exactly four distinct non-trivial lines, or $p$ is incident with a non-trivial line $\ell$ which contains a point $p'$ which is incident with four distinct non-trivial lines.
\end{enumerate} 
Notice that for a point $p \in P_{\Gamma}$ we have:
\begin{enumerate}[$(\star_2)$]
\item $P_{\Gamma} \models \varphi(p)$ if and only if $p \in P_*$.
\end{enumerate}
In fact, if the point $p \in P_*$, then either it is the point $q$, in which case there are four distinct non-trivial lines which are incident with it, or we can find a non-trivial line $\ell$ which is incident with the point $p$ and contains the point $p_3$ (this is clear by inspection of Figure \ref{figure1}). On the other hand, if the point $p \not\in P_*$, then it is either $p_{(\alpha, 0)}$, $p_{(\alpha, 1)}$, $p_{(\alpha, 2)}$, or $p_e$, for some $\alpha < \lambda$ and $e \in E_{\Gamma}$. Notice now that: 
\begin{enumerate}[$(\star_3)$]
	\item if $p = p_{(\alpha, 0)}$, then $p$ is incident with exactly two non-trivial lines, namely the lines $p_2 \vee 1'$ and $p_{(\alpha, 0)} \vee p_{(\alpha, 1)}$;
\end{enumerate} 
\begin{enumerate}[$(\star_4)$]
	\item if $p = p_{(\alpha, 1)}$, then $p$ is incident with exactly two non-trivial lines, namely the lines $0 \vee 1'$ and $p_{(\alpha, 0)} \vee p_{(\alpha, 1)}$;
\end{enumerate}
\begin{enumerate}[$(\star_5)$]
	\item the point $p_2$ is incident with exactly two non-trivial lines, namely the line $p_2 \vee 0$ and the line $p_2 \vee 1'$; the point $0$ is incident with exactly three non-trivial lines, namely the lines $p_2 \vee 0$, $0 \vee 0'$ and $0 \vee 1'$; the point $1'$ is incident with exactly three non trivial lines, namely the lines $1' \vee 0'$, $1' \vee 1_0$ and $1' \vee 2_1$;
\end{enumerate} 
\begin{enumerate}[$(\star_6)$]
	\item if $p = p_{e}$ and $e = \{ v_{\delta}, v_{\alpha} \}$, then $p_{e}$ is incident with exactly two non-trivial lines, namely the lines $p_{(\delta, 0)} \vee p_{(\delta, 1)}$ and $p_{(\alpha, 0)} \vee p_{(\alpha, 1)}$;
\end{enumerate}
\begin{enumerate}[$(\star_7)$]
	\item for $\alpha < \lambda$, the set of points incident with the line $p_{(\alpha, 0)} \vee p_{(\alpha, 1)}$ is:
	$$\{ p_{(\alpha, 0)}, p_{(\alpha, 1)}, p_{(\alpha, 2)} \} \cup \{ p_{e} : p_{\alpha} \in e \in E_{\Gamma} \};$$
\end{enumerate} 
\begin{enumerate}[$(\star_8)$]
	\item if  $\alpha \neq \beta < \lambda$, then $p_{(\alpha, 0)} \vee p_{(\beta, 1)}$ is a trivial line.
\end{enumerate}      
Thus, by $(\star_3)$-$(\star_8)$, it is clear that for $p \notin P_*$ we have that $P_{\Gamma} \not\models \varphi(p)$.

\smallskip
\noindent We now prove (\ref{iso_invariance}). Let $f: P_{\Gamma_1} \cong P_{\Gamma_2}$, $|\Gamma_1| = \lambda$ and, for $i = 1, 2$, let the set of points of $P_{\Gamma_i}$ be:
$$\{ (p, i) : p \in P_* \} \cup \{ p^i_{(\alpha, j)} : j < 3, \alpha < \lambda \} \cup \{ p^i_e : e \in E_{\Gamma_i} \},$$
(cf. $(*)$ above). By $(\star)_2$, we have that $f$ restricted to
$\{ (p, 1) : p \in P_* \}$ is an isomorphism from $\{ (p, 1) : p \in P_* \}$ onto $\{ (p, 2) : p \in P_* \}$, and so, as $P_*$ is rigid, for every $p \in P_*$ we have that $f((p, 1)) = (p, 2)$. In particular, the line $(p_2, 1) \vee (1', 1)$ is mapped to the line $(p_2, 2) \vee (1', 2)$, and the line $(0, 1) \vee (1', 1)$ is mapped to the line $(0, 2) \vee (1', 2)$. Thus,  $f$ maps $\{ p^1_{(\alpha, 0)} : \alpha < \lambda \}$ onto $\{ p^2_{(\alpha, 0)} : \alpha < \lambda \}$ and $\{ p^1_{(\alpha, 1)} : \alpha < \lambda \}$ onto $\{ p^2_{(\alpha, 1)} : \alpha < \lambda \}$.
Also, by $(\star_6)$, $f$ maps $\{ p^1_{(\alpha, 2)} : \alpha < \lambda \}$ onto $\{ p^2_{(\alpha, 2)} : \alpha < \lambda \}$. Finally, if  $\alpha \neq \beta < \lambda$ and $f(p^1_{(\alpha, 0)}) = p^1_{(\beta, 0)}$, then $f(p^1_{(\alpha, 1)}) = p^1_{(\beta, 1)}$, since otherwise $f$ would send the non-trivial line $p^1_{(\alpha, 0)} \vee p^1_{(\alpha, 1)}$ to a trivial line (cf. $(\star_8)$).  
Thus, $f$ induces a bijection:
$$f_* : \{ p^1_{(\alpha, 0)} \vee p^1_{(\alpha, 1)} : \alpha < \lambda \} \rightarrow \{ p^2_{(\alpha, 0)} \vee p^2_{(\alpha, 1)} : \alpha < \lambda \}.$$
Hence, by $(\star)_1$, the map $f_*$ induces an isomorphism from $\Gamma_1$ onto $\Gamma_2$, since clearly the isomorphism $f$ sends pairs of intersecting lines to pairs of intersecting lines. Finally, item (\ref{auto}) is clear from the proof of item (\ref{iso_invariance}).
\end{proof}


\section{Proof of Theorem \ref{main_theorem_proj}}

	In this section we prove Theorem \ref{main_theorem_proj}.

\begin{notation} We denote by $Q$ the plane represented in the matrix in Figure \ref{myfigure}, where the letters occurring in the matrix represent the points of $Q$, and the columns of the matrix represent the lines of $Q$.
The plane $Q$ is taken from \cite{projective} (cf. \cite[Diagram 1]{projective}), where it is attributed to S. Ditor.

	\begin{figure}[htb]
\begin{center}
\[ \begin{bmatrix}
    a & c & e & a & b & d & d & c & e & a & b \\
    b & n & o & f & k & n & o & k & m & k & n \\
    c & l & l & g & l & k & m & g & g & o & o \\
    d & f & f & h & m & f & h &   &   &   &   \\
    e &   &   &  &    &   &   &   &   &   &   
  \end{bmatrix}
\]
\caption{The plane $Q$.}\label{myfigure}
\end{center}
\end{figure}
\end{notation}

	\begin{strategy}\label{strategy} In proving Theorem \ref{main_theorem_proj} we will follow the following strategy:
	\begin{enumerate}[(1)]
	\item for $\Gamma$ an infinite graph, consider the $P_{\Gamma}$ of Theorem \ref{main_theorem} and extend it to a $P^+_{\Gamma}$ adding independent copies of the plane $Q$ (cf. Figure \ref{myfigure}) at each point not in a finite confined subplane (cf. Definition \ref{def_conf}(2)), and then adding independent copies of $Q$ at each line not in a finite confined subplane, repeating this process for lines $\omega$-many times (for points one application of the process suffices);
	\item observe that, restricted to structures with domain $\omega$, the set of $P_{\Gamma}$'s is Borel and that the map $\Gamma \mapsto P_{\Gamma} \mapsto P^+_{\Gamma}$ is Borel; 
	\item prove that $\Gamma \mapsto P^+_{\Gamma}$ is isomorphism invariant and that $Aut(\Gamma) \cong Aut(P^+_{\Gamma})$;
	\item observe that, restricted to structures with domain $\omega$, the map $P \mapsto F(P)$ (cf. Definition \ref{def_free_ext}) is Borel;
	\item consider the free projective extension $F(P^+_{\Gamma})$ of $P^+_{\Gamma}$, and use Fact \ref{desargue_fact} for non-Desarguesianess, Fact \ref{piper_fact1} for isomorphism invariance, and Fact \ref{piper_fact2} for:
	$$Aut(\Gamma) \cong Aut(F(P^+_{\Gamma})).$$
\end{enumerate}
\end{strategy}

	First of all we deal with Strategy \ref{strategy}(4):

	\begin{lemma}\label{Borel_th} Restricted to structures with domain $\omega$, the map $P \mapsto F(P)$ associating to each plane its free projective extension is a Borel map.
\end{lemma}

	\begin{proof} Essentially as in the proof of Theorem \ref{main_theorem}.
\end{proof}

	Before proving Theorem \ref{main_theorem_proj} we isolate two constructions which will be crucially used in implementing Strategy \ref{strategy}(1).
	
	\begin{construction}\label{construction1} Let $P$ be a plane and $p$ a point of $P$. We define $P(p, Q, a)$ as the extension of $P$ obtained by adding an independent copy of $Q$ to $P$ identifying the point $p$ of $P$ and the point $a$ of $Q$, in such a way that if $p'$ is a point of $P$ different than $p$, and $q$ is a point of $Q$ different than $a$, then $p' \vee q$ is a trivial line.
\end{construction}

	\begin{construction}\label{construction2} Let $P$ be a plane and $\ell$ a line of $P$. We define $P(\ell, Q, a \vee b)$ as the extension of $P$ obtained by adding an independent copy of $Q$ to $P$ identifying the line $\ell$ of $P$ and the line $a \vee b$ of $Q$, in such a way that if $p'$ is a point of $P$ not on $\ell$, and $q$ is a point of $Q$ not on $a \vee b$, then $p' \vee q$ is a trivial line.
\end{construction}

	\begin{remark} The construction of $P(p, Q, a)$ and $P(\ell, Q, a \vee b)$ from $P$ can be formally justified using Fact \ref{fact}. We elaborate on this:
	\begin{enumerate}[(i)]
	\item Concerning the case $P(p, Q, a)$. Add two generic points\footnote{I.e. $b$ and $f$ are not incident with any line of $P$.} $b$ and $f$ to $P$, corresponding to the points $b$ and $f$ of $Q$. Then $\langle p, b, f \rangle_P \cong \langle a, b, f \rangle_Q$ is a copy of the simple matroid of rank $3$ and size $3$. Now construct a copy of $Q$ in $P$ from $\{ p, b, f \}$ point by point, following how $Q$ is constructed from $\{ a, b, f \}$ point by point. Notice that the order in which we do this does not matter.
	\item Concerning the case $P(\ell, Q, a \vee b)$. Firs of all, let $p$ and $q$ be points of $P$ such that $p \vee q = \ell$. Now, add one generic point\footnote{I.e. $f$ is not incident with any line of $P$.} $f$ to $P$, corresponding to the point $f$ of $Q$. Then $\langle p, q, f \rangle_P \cong \langle a, b, f \rangle_Q$ is a copy of the simple matroid of rank $3$ and size $3$. Now construct a copy of $Q$ in $P$ from $\{ p, q, f \}$ point by point, following how $Q$ is constructed from $\{ a, b, f \}$ point by point. Notice that the choice of $p$ and $q$ does not matter, as well as the order in which we construct the copy of $Q$ in $P$ from $\{ p, q, f \}$, as observed also in (i).
\end{enumerate}	
\end{remark}

\begin{proof}[Proof of Theorem \ref{main_theorem_proj}] We follow the strategy delineated in Strategy \ref{strategy}. Let $\Gamma$ be an infinite graph and $P_{\Gamma}$ be the respective plane from Theorem \ref{main_theorem}. We define $P^+_{\Gamma}$ as the union of a chain of planes $(P^n_{\Gamma}: n < \omega)$, defined by induction on $n < \omega$. 
\newline \underline{$n = 0$}. Let $\{ p_{\alpha} : 0 < \alpha < \kappa \}$ be an injective enumeration of the points of $P_{\Gamma}$ not in a finite confined configuration (notice that there infinitely many such point in $P_{\Gamma}$). Let then:
\begin{enumerate}[(i)]
	\item $P^{(0, 0)}_{\Gamma} = P_{\Gamma}$;
	\item $P^{(0, \alpha)}_{\Gamma} = P^{(0, \alpha-1)}_{\Gamma}(p_{\alpha}, Q, a)$, for $0 <\alpha < \kappa$ successor (cf. Construction \ref{construction1});
	\item $P^{(0, \alpha)}_{\Gamma} = \bigcup_{\beta < \alpha} P^{(0, \beta)}_{\Gamma}$, for $\alpha$ limit;
	\item $P^0_{\Gamma} = \bigcup_{\alpha < \kappa} P^{(0, \alpha)}_{\Gamma}$.
\end{enumerate}
(Notice that the choice of the enumeration $\{ p_{\alpha} : 0 < \alpha < \kappa \}$ does not matter, since the copies of $Q$ that we add at every point are independent. In particular, in the countable case we can take the enumeration to be Borel. Furthermore, we now have that every point of $P^0_{\Gamma}$ is contained in a finite confined subplane of $P^0_{\Gamma}$.)
\newline \underline{$n > 0$}. Let $\{ \ell_{\alpha} : 0 < \alpha < \mu \}$ be an injective enumeration of the lines of $P^{n-1}_{\Gamma}$ not in a finite confined configuration (notice that there infinitely many such lines in $P^{n-1}_{\Gamma}$, this is true for $n - 1 = 0$, and it is preserved by the induction). Let then:
\begin{enumerate}[(i)]
	\item $P^{(n, 0)}_{\Gamma} = P^{n-1}_{\Gamma}$;
	\item $P^{(n, \alpha)}_{\Gamma} = P^{(0, \alpha-1)}_{\Gamma}(\ell_{\alpha}, Q, a \vee b)$, for $0 <\alpha < \mu$ successor (cf. Construction \ref{construction2});
	\item $P^{(n, \alpha)}_{\Gamma} = \bigcup_{\beta < \alpha} P^{(n, \beta)}_{\Gamma}$, for $\alpha$ limit;
	\item $P^n_{\Gamma} = \bigcup_{\alpha < \mu} P^{(n, \alpha)}_{\Gamma}$.
\end{enumerate}
(Notice that also in this case the choice of the enumeration $\{ \ell_{\alpha} : 0 < \alpha < \mu \}$ does not matter, since the copies of $Q$ that we add at every line are independent. In particular, in the countable case we can take the enumeration to be Borel. Furthermore, inductively, we maintain the condition that every point of $P^n_{\Gamma}$ is contained in a finite confined subplane of $P^n_{\Gamma}$ (although this is not true for lines).)
\newline Let then $P^+_{\Gamma} = \bigcup_{n < \omega} P^n_{\Gamma}$. First of all, observe that the class of $P_{\Gamma}$'s ($P_\Gamma$ and $\Gamma$ with domain $\omega$) is Borel, since the appropriate restriction of the map $\Gamma \mapsto P_{\Gamma}$ is injective, in fact if $\Gamma  \neq \Gamma'$, then there are $n \neq k \in \omega$ such that $n E_{\Gamma} k$ and $n \!\!\not\!\!E_{\Gamma'} k$ (by symmetry) and so in $P_{\Gamma}$ the (codes of the) lines $p_{(n, 0)} \vee p_{(n, 1)}$ and $p_{(k, 0)} \vee p_{(k, 1)}$ are incident while in $\Gamma'$ they are parallel. Furthermore, by the uniformity of the construction, the map $P^+_{\Gamma}$ from $P_{\Gamma}$ is Borel, when restricted to structures with domain $\omega$. Also, notice that the plane $P^+_{\Gamma}$ is confined and not projective, and so if we manage to complete Strategy \ref{strategy}(3), then by Lemma \ref{Borel_th} and Facts \ref{desargue_fact}, \ref{piper_fact1} and \ref{piper_fact2} we are done (as delineated in Strategy \ref{strategy}(4-5)). We are then only left with Strategy \ref{strategy}(3). To this extent notice that:
\begin{enumerate}[$(\star_1)$]
	\item the points from $P^+_{\Gamma}$ which are incident with at least four non-trivial lines are exactly the points of $P_{\Gamma}$.
\end{enumerate} 
Thus, from $(\star_1)$ it is clear that if $P^+_{\Gamma_1} \cong P^+_{\Gamma_2}$, then $P_{\Gamma_1} \cong P_{\Gamma_2}$, which in turn implies that $\Gamma_1 \cong \Gamma_2$ (cf. Theorem \ref{main_theorem}(5)). Furthermore, using again $(\star_1)$, and the fact that by \cite[Lemma 1]{projective} the plane $Q$ has trivial automorphism group, it is easy to see that:
\begin{enumerate}[$(\star_2)$]
	\item every $f \in Aut(P^+_{\Gamma})$ is induced by a $f^- \in Aut(P_{\Gamma})$;
\end{enumerate} 
\begin{enumerate}[$(\star_3)$]
	\item every $f \in Aut(P_{\Gamma})$ extends uniquely to a $f^+ \in Aut(P^+_{\Gamma})$.
\end{enumerate} 
Thus, we have that $Aut(P^+_{\Gamma}) \cong Aut(P_{\Gamma}) \cong Aut(\Gamma)$, by Theorem \ref{main_theorem}(5).
\end{proof}

\section{Other proofs}

Corollary \ref{main_cor2} is a standard consequence of Theorems \ref{main_theorem} and \ref{main_theorem_proj} (see e.g. \cite{friedman} and \cite{gao} for an overview on Borel completeness, and \cite{ulm_inv_paper} for Ulm invariants). Also, Corollary \ref{main_cor} follows from Theorems \ref{main_theorem} and \ref{main_theorem_proj} and the following fact:

	\begin{fact} 
\begin{enumerate}[(1)]
	\item For every finite structure $M$ (in the sense of model theory) there exists a finite graph $\Gamma_M$ such that $Aut(\Gamma_M) \cong Aut(M)$.
	\item For every infinite structure $M$ (in the sense of model theory) there exists a graph $\Gamma_M$ of the same cardinality of $M$ such that $Aut(\Gamma_M) \cong Aut(M)$.
\end{enumerate}
\end{fact}

	Finally, we prove Theorem \ref{des_theorem}. To this extent we need the following fact.

	\begin{fact}[{\cite[3.2]{friedman}}]\label{field_fact} The class of countable fields is Borel complete.
\end{fact}

	\begin{proof}[Proof of Theorem \ref{des_theorem}] Immediate from Facts \ref{pappian_field_fact}, \ref{pappian_fact} and \ref{field_fact}.
\end{proof}

\end{document}